\documentclass[11pt]{amsart}
\usepackage{amsmath}
\usepackage{amsfonts}
\usepackage{amssymb}
\usepackage{color}
\usepackage{enumerate}
\usepackage{graphicx}
\usepackage{hyperref}

\newtheorem{thm}{Theorem}

\newtheorem{cor}[thm]{Corollary}

\newtheorem{remark}[thm]{Remark}
\newtheorem{lemma}[thm]{Lemma}
\newtheorem{prop}[thm]{Proposition}

\newcommand{\Tr}{{\rm Tr}}
\newcommand{\bb}[1]{\mathbb{#1}}
\newcommand{\cl}[1]{\mathcal{#1}}

\newcommand{\mnorm}[1]{%
\left\vert\kern-0.9pt\left\vert\kern-0.9pt\left\vert #1
\right\vert\kern-0.9pt\right\vert\kern-0.9pt\right\vert}

\newcommand{\bigmnorm}[1]{%
\big\vert\kern-0.9pt\big\vert\kern-0.9pt\big\vert #1
\big\vert\kern-0.9pt\big\vert\kern-0.9pt\big\vert}

\title[Mapping Cones are Operator Systems]{Mapping Cones are Operator Systems}

\author[N.~Johnston]{Nathaniel Johnston}
\address{Department of Mathematics and Statistics, University of Guelph, Guelph, Ontario N1G 2W1, Canada}
\email{njohns01@uoguelph.ca}

\author[E.~St\o rmer]{Erling St\o rmer}
\address{Department of Mathematics, University of Oslo, P.O. Box 1053 Blindern, NO-0316 Oslo, Norway}
\email{erlings@math.uio.no}

\begin{document}

\begin{abstract}
We investigate the relationship between mapping cones and matrix ordered $*$-vector spaces (i.e., abstract operator systems). We show that to every mapping cone there is an associated operator system on the space of $n$-by-$n$ complex matrices, and furthermore we show that the associated operator system is unique and has a certain homogeneity property. Conversely, we show that the cone of completely positive maps on any operator system with that homogeneity property is a mapping cone. We also consider several related problems, such as characterizing cones that are closed under composition on the right by completely positive maps, and cones that are also semigroups, in terms of operator systems.\medskip

\noindent {\bf Keywords:} operator systems, mapping cones, dual cones, positive maps\medskip

\noindent {\bf AMS Subject Classifications:} 15B48, 47D03, 47L99
\end{abstract}

\maketitle

%%%%%%%%%%%%%%%%%%%%%%%%%%%%%%%%%%%%%%%%
\section{Introduction}
%%%%%%%%%%%%%%%%%%%%%%%%%%%%%%%%%%%%%%%%

In operator theory, some of the most important families of linear maps are the positive and $k$-positive maps, and their dual cones \cite{SSZ09} of superpositive and $k$-superpositive maps. These sets of maps are all specific examples of mapping cones \cite{S86}, which are closed cones of positive maps that are invariant under left and right composition by completely positive maps -- a property of $k$-positive and $k$-superpositive maps that is easily verified.

It has recently been shown \cite{PTT10,Xthesis,JKPP11} that the $k$-positive and $k$-superpositive maps can be seen as the completely positive maps on certain natural operator system structures. We thus have two settings, seemingly very different, that give rise to the familiar cones of $k$-positive and $k$-superpositive maps. A natural question that arises is whether this is simply coincidence, or if there is indeed a fundamental link between mapping cones and operator systems.

In this work, we show that there is indeed an extremely strong connection between mapping cones and operator systems. In fact, we show that there is a bijection between mapping cones and operator systems with a property that we refer to as super-homogeneity. If we remove the super-homogeneity property, then the bijection is no longer with mapping cones but rather with cones that are only closed under right (but not necessarily left) composition by completely positive maps. We also answer some related questions involving semigroup cones of positive maps.

The paper is organized as follows. In Section~\ref{sec:cone_prelims} we introduce much of our notation and present the basics of cones of linear maps on complex matrices. In Section~\ref{sec:os_prelims} we present abstract operator systems and derive a simple uniqueness property in the finite-dimensional setting that we are interested in. In Section~\ref{sec:opSysRight} we present and prove our most general result for right-CP-invariant cones, which shows their intimate link with operator systems, and in Section~\ref{sec:opSys} we investigate the special case of mapping cones and the kinds of operator systems that they give rise to. We close in Section~\ref{sec:semigroups} by exploring some properties of mapping cones that have the additional property of being semigroups, and we see that they too can be seen as arising from operator systems.

%%%%%%%%%%%%%%%%%%%%%%%%%%%%%%%%%%%%%%%%
\section{Cones of Positive Maps}\label{sec:cone_prelims}
%%%%%%%%%%%%%%%%%%%%%%%%%%%%%%%%%%%%%%%%

If $\cl{H}$ is a (finite-dimensional) Hilbert space and $\cl{L}(\cl{H})$ is the space of linear maps on $\cl{H}$, then a map $\Phi : \cl{L}(\cl{H}) \rightarrow \cl{L}(\cl{H})$ is said to be \emph{positive} if $\Phi(X) \in \cl{L}(\cl{H})^{+}$ whenever $X \in \cl{L}(\cl{H})^{+}$, \emph{$k$-positive} if $id_k \otimes \Phi$ is positive, and \emph{completely positive} if $\Phi$ is $k$-positive for all $k \in \bb{N}$. If $A \in \cl{L}(\cl{H})$ then the map ${\rm Ad}_A : \cl{L}(\cl{H}) \rightarrow \cl{L}(\cl{H})$ defined by ${\rm Ad}_A(X) \equiv A^* XA$ is completely positive, and conversely every completely positive map can be written as a sum of maps of the form ${\rm Ad}_{A_i}$ for some $\{A_i\} \in \cl{L}(\cl{H})$ \cite{NC00,C75}.

Given a fixed orthonormal basis $\{{\mathbf e_i}\}_{i=1}^n$ of $\cl{H}$, the \emph{Choi-Jamio\l kowski isomorphism} \cite{J72} associates a linear map $\Phi : \cl{L}(\cl{H}) \rightarrow \cl{L}(\cl{H})$ with the operator $C_{\Phi} := (id_n \otimes \Phi)(E) \in \cl{L}(\cl{H}) \otimes \cl{L}(\cl{H})$, where $E := \sum_{i,j=1}^n {\mathbf e_i}{\mathbf e_j}^* \otimes {\mathbf e_i}{\mathbf e_j}^*$ (${\mathbf e_j}^*$ is the dual vector of ${\mathbf e_j}$, ${\mathbf e_i}{\mathbf e_j}^*$ is the outer product of ${\mathbf e_i}$ and ${\mathbf e_j}$, and $C_{\Phi}$ is called the \emph{Choi matrix} of $\Phi$). For us, it will be useful to know that $\Phi$ is completely positive if and only if $C_{\Phi}$ is positive, and $\Phi$ is positive if and only if $C_{\Phi}$ is \emph{block-positive} -- i.e., $({\mathbf v}^* \otimes {\mathbf w}^*)C_{\Phi}({\mathbf v} \otimes {\mathbf w}) \geq 0$ for all ${\mathbf v},{\mathbf w} \in \cl{H}$. Given a cone of positive maps $\cl{C}$, we define $C_{\cl{C}} := \{ C_{\Phi} : \Phi \in \cl{C} \}$ and $\cl{C}^{\dagger} := \{ \Phi^{\dagger} : \Phi \in \cl{C} \}$, where $\Phi^{\dagger} : \cl{L}(\cl{H}) \rightarrow \cl{L}(\cl{H})$ is the adjoint map defined via the Hilbert-Schmidt inner product so that $\Tr(\Phi(X)Y) = \Tr(X\Phi^{\dagger}(Y))$ for all $X,Y \in \cl{L}(\cl{H})$.

A \emph{mapping cone} \cite{S86} is a nonzero closed cone $\cl{C}$ of positive maps from $\cl{L}(\cl{H})$ into itself with the property that $\Phi \circ \Omega \circ \Psi \in \cl{C}$ whenever $\Omega \in \cl{C}$ and $\Phi,\Psi : \cl{L}(\cl{H}) \rightarrow \cl{L}(\cl{H})$ are completely positive. For the remainder of this work, we will generally assume all cones to be convex, though we will still specify if the distinction is important or there is the possibility of confusion. By linearity, it is enough that ${\rm Ad}_A \circ \Omega \circ {\rm Ad}_B \in \cl{C}$ whenever $\Omega \in \cl{C}$ and $A,B \in \cl{L}(\cl{H})$ for a convex cone $\cl{C}$ to be a mapping cone. It will also sometimes be useful for us to consider (not necessarily closed) cones $\cl{C}$ such that $\Omega \circ \Psi \in \cl{C}$ whenever $\Omega \in \cl{C}$ and $\Psi : \cl{L}(\cl{H}) \rightarrow \cl{L}(\cl{H})$ is completely positive -- that is, cones that are closed under right-composition, but not necessarily left-composition, by completely positive maps. We will call such cones \emph{right-CP-invariant}. Left-CP-invariant cones can be defined analogously, and it is clear that $\cl{C}$ is right-CP-invariant if and only if $\cl{C}^{\dagger}$ is left-CP-invariant.

The \emph{dual cone} $\cl{C}^{\circ}$ of a cone $\cl{C} \subseteq \cl{L}(\cl{H})$ of Hermitian operators is defined via the Hilbert-Schmidt inner product as
\begin{align*}
	\cl{C}^{\circ} := \{ Y \in \cl{L}(\cl{H}) : \Tr(XY) \geq 0 \ \text{ for all } X \in \cl{C} \}.
\end{align*}
Similarly, the dual cone $\cl{C}^{\circ}$ of a cone $\cl{C}$ of maps on $\cl{L}(\cl{H})$ is defined via the Choi-Jamio\l kowski isomorphism as $\cl{C}^{\circ} := \{ \Psi : \cl{L}(\cl{H}) \rightarrow \cl{L}(\cl{H}) : \Tr(C_{\Phi}C_{\Psi}) \geq 0 \ \text{ for all } \Phi \in \cl{C} \}$. We note that for convex cones $\cl{C} \subseteq \cl{L}(\cl{H})$, we have $\cl{C}^{\circ \circ} = \overline{\cl{C}}$ -- the closure of $\cl{C}$. This fact is well-known in convex analysis and follows easily from \cite[Theorem 14.1]{R97} or \cite[Theorem 3.4.3]{GY02}, for example.

Throughout the rest of this work, we will associate the $n$-dimensional Hilbert space $\cl{H}$ with $\bb{C}^n$ and $\cl{L}(\cl{H})$ with the space of $n \times n$ complex matrices $M_n$, both for simplicity and to be consistent with standard operator system notation. Then $\cl{L}(M_n)$ denotes the set of linear maps from $M_n$ into itself, $\cl{P}(M_n)$ denotes the set of positive maps on $M_n$, $\cl{P}_k(M_n)$ denotes the set of $k$-positive maps on $M_n$, and $\cl{CP}(M_n)$ the set of completely positive maps on $M_n$. We let $P_n \subseteq M_n \otimes M_n$ denote the cone of block-positive operators, and $S_n := P_n^{\circ}$ is the cone of separable operators \cite{SSZ09} -- operators $X \in M_n \otimes M_n$ that can be written in the form
\begin{align*}
	X = \sum_i Y_i \otimes Z_i \ \text{ for positive semidefinite } \{Y_i\}, \{Z_i\} \in M_n.
\end{align*}

Associated to the cone of separable operators via the Choi-Jamio\l kowski isomorphism is the cone of \emph{superpositive} maps $\cl{S}(M_n)$ (sometimes called \emph{entanglement-breaking maps} \cite{HSR03}). Similarly, the cone of \emph{$k$-superpositive} maps is the dual cone of the cone of $k$-positive maps: $\cl{S}_k(M_n) := \cl{P}_k(M_n)^{\circ}$. 

We close this section with a simple lemma (which also appeared as \cite[Lemma 3]{S11} with a different proof) that allows us to relate the Choi matrices of $\Phi$ and $\Phi^{\dagger}$. Note that in the particularly important case $\Phi = {\rm Ad}_A$, the lemma says that $(id_n \otimes {\rm Ad}_A)(E) = ({\rm Ad}_{A^T} \otimes id_n)(E)$, where $T$ denotes the transpose map.
\begin{lemma}\label{lem:left_right_choi}
	Let $\Phi : M_n \rightarrow M_n$. Then $(id_n \otimes \Phi)(E) = ((T \circ \Phi^\dagger \circ T) \otimes id_n)(E)$.
\end{lemma}
\begin{proof}
	Throughout this proof, by a vectorization ${\rm vec}(X)$ of a matrix $X$, we mean the vector in $\bb{C}^{n} \otimes \bb{C}^{n} \cong \bb{C}^{n^2}$ obtained from $X \in M_n$ by stacking the columns of $X$ on top of each other, starting with the leftmost column. Use the singular value decomposition to write $C_{\Phi} = \sum_i {\mathbf v_i}{\mathbf w_i}^*$. It is easily verified that for any $X \in M_n$, ${\rm vec}(X^T) = F{\rm vec}(X)$, where $F$ is the ``swap'' or ``flip'' operator that acts on elementary tensors as $F({\mathbf a} \otimes {\mathbf b}) = {\mathbf b} \otimes {\mathbf a}$. The result follows from recalling (see \cite{C75} or \cite[Proposition 6.2]{WLN}, for example) that we can write $\Phi(X) = \sum_i A_i X B_i^{*}$ and $\Phi^{\dagger}(Y) = \sum_i B_i^{*} Y A_i$, where ${\rm vec}(A_i) = {\mathbf v_i}$ and ${\rm vec}(B_i) = {\mathbf w_i}$.
\end{proof}

%%%%%%%%%%%%%%%%%%%%%%%%%%%%%%%%%%%%%%%%
\section{Operator Systems on $M_n$}\label{sec:os_prelims}
%%%%%%%%%%%%%%%%%%%%%%%%%%%%%%%%%%%%%%%%

An (abstract) operator system on $M_n$ is a family of convex cones $\{C_m\}_{m=1}^{\infty} \subseteq M_m \otimes M_n$ that satisfy the following two properties:
\begin{itemize}
	\item $C_1 = M_n^{+}$, the cone of positive semidefinite elements of $M_n$; and
	\item for each $m_1,m_2 \in \bb{N}$ and $A \in M_{m_1,m_2}$ we have $({\rm Ad}_{A} \otimes id_n)(C_{m_1}) \subseteq C_{m_2}$.
\end{itemize}
Abstract operator systems can be defined more generally as matrix ordered $*$-vector spaces on any Archimedean $*$-ordered vector space $V$, but the above definition with $V = M_n$ is much simpler and suited to our particular needs. The interested reader is directed to \cite[Chapter 13]{Paulsentext} for a more thorough treatment of general abstract operator systems. The fact that matrix ordered $*$-vector spaces can be thought of as operator systems follows from the work of Choi and Effros \cite{CE77}.

\begin{remark}{\rm
	Abstract operator systems typically are defined with two additional requirements that we have not mentioned:
	\begin{itemize}
		\item $C_m \cap -C_m = \{0\}$ for each $m \in \bb{N}$; and
		\item for every $m \in \bb{N}$ and $X = X^* \in M_m \otimes M_n$, there exists $r > 0$ such that $rI + X \in C_m$.
	\end{itemize}
	Both of these conditions follow for free from the fact that, in our setting, $C_1 = M_n^{+}$.
	
	To see that the first property holds, notice that $C_1 \cap -C_1 = \{0\}$, and suppose that $X \in C_m \cap -C_m$ for some $m \geq 2$. Then $({\rm Ad}_{A} \otimes id_n)(X) \in C_1 \cap -C_1$ for any $A \in M_{m,1}$. Because $C_1 \cap -C_1 = \{0\}$, it follows that $({\mathbf v}^*\otimes{\mathbf w}^*)X({\mathbf v}\otimes{\mathbf w}) = 0$ for all ${\mathbf v} \in \bb{C}^m$, ${\mathbf w} \in \bb{C}^n$. It follows (via \cite[Lemma 2.1]{J11}, for example) that $X = 0$, so $C_m \cap -C_m = \{0\}$ for all $m \in \bb{N}$.
	
	The second property holds because the smallest family of cones on $M_n$ such that $({\rm Ad}_{A} \otimes id_n)(C_{m_1}) \subseteq C_{m_2}$ for all $m_1,m_2 \in \bb{N}$ are the cones of separable operators in $M_m \otimes M_n$ \cite[Theorem 5]{JKPP11}. It is well-known that there always exists $r > 0$ such that $rI + X$ is separable \cite{GB02}, so the same $r$ ensures that $rI + X \in C_m$.
}\end{remark}

One particularly important operator system is the one constructed by associating $M_m \otimes M_n$ with $M_{mn}$ in the natural way and letting $C_m \subseteq M_m \otimes M_n$ be the cones of positive semidefinite operators. We will denote this operator system simply by $M_n$, and it will be clear from context whether we mean the operator system $M_n$ or simply the set $M_n$ without regard to any family of cones. Other operator systems on $M_n$ will be denoted like $O(M_n)$ in order to avoid confusion with the operator system $M_n$ itself.

If $O_1(M_n)$ and $O_2(M_n)$ are two operator systems defined by the cones $\{C_m\}_{m=1}^{\infty}$ and $\{D_m\}_{m=1}^{\infty}$ respectively, then a map $\Phi : M_n \rightarrow M_n$ is said to be completely positive from $O_1(M_n)$ to $O_2(M_n)$ if $(id_m \otimes \Phi)(C_m) \subseteq D_m$ for all $m \in \bb{N}$. The set of maps that are completely positive from $O_1(M_n)$ to $O_2(M_n)$ is denoted by $\cl{CP}(O_1(M_n),O_2(M_n))$, or simply $\cl{CP}(O(M_n))$ if the target operator system equals the source operator system. It will often be useful for us to consider operator systems with the additional property that $(id_m \otimes {\rm Ad}_{B})(C_{m}) \subseteq C_{m}$ for each $m \in \bb{N}$ and $B \in M_{n}$ -- a property that is equivalent to the fact $\cl{CP}(M_n) \subseteq \cl{CP}(O(M_n))$. We will call operator systems with this property \emph{super-homogeneous}.

We now present a result that shows that operator systems on $M_n$ are in fact characterized completely by their $n^{th}$ cone. That is, there is a unique way to construct an operator system given an appropriate cone $C_n \subseteq M_n \otimes M_n$.
\begin{prop}\label{prop:cn_to_os}
	Let $C_n \subseteq M_n \otimes M_n$ be a convex cone such that $S_n \subseteq C_n \subseteq P_n$ and $({\rm Ad}_A \otimes id_n)(C_n) \subseteq C_n$ for all $A \in M_n$. Then there exists a unique family of cones $\{C_m\}_{m\neq n}$ such that $\{C_m\}_{m=1}^{\infty}$ defines an operator system on $M_n$, given by
	\begin{align*}
		C_m := \big\{ \sum_i({\rm Ad}_{A_i} \otimes id_n)(X) : A_i \in M_{n,m} \ \forall \, i, X \in C_n \big\}.
	\end{align*}
	Furthermore, the operator system is super-homogeneous if and only if $(id_n \otimes {\rm Ad}_{B})(C_n) \subseteq C_n$ for all $B \in M_n$.
\end{prop}
\begin{proof}
	We first prove that the family of convex cones given by the proposition do indeed define an operator system. We first show that $({\rm Ad}_B \otimes id_n)(Y) \in C_{m_2}$ for any $m_1,m_2 \in \bb{N}$, $Y \in C_{m_1}$, and $B \in M_{m_1,m_2}$. This is true from the definition of $C_m$ if $m_1 = n$. If $m_1 \neq n$ then write $Y = \sum_i({\rm Ad}_{A_i} \otimes id_n)(X)$ for some $X \in C_n$ and $\{A_i\} \subset M_{n,m_1}$. Then $A_i B \in M_{n,m_2}$ for all $i$, so
	\begin{align*}
		({\rm Ad}_B \otimes id_n)(Y) = \sum_i({\rm Ad}_{A_i B} \otimes id_n)(X) \in C_{m_2}.
	\end{align*}
	
	We now show that $C_1 = M_n^{+}$. For any ${\mathbf v} \in \bb{C}^n$, note that ${\mathbf v}{\mathbf v}^* \otimes X \in S_n$ if and only if $X \in M_n^{+}$, and similarly ${\mathbf v}{\mathbf v}^* \otimes X \in P_n$ if and only if $X \in M_n^{+}$. It follows that ${\mathbf v}{\mathbf v}^* \otimes X \in C_n$ if and only if $X \in M_n^{+}$. Then $C_1 \supseteq \{ ({\rm Ad}_{A} \otimes id_n)({\mathbf v}{\mathbf v}^* \otimes X) : A \in M_{n,1}, X \in M_n^{+} \big\} = M_n^{+}$, where we have identified $\bb{R}_+ \otimes M_n^+$ with $M_n^+$. The opposite inclusion follows simply from noting that if $X \in C_1$ and ${\mathbf v} \in \bb{C}^n$ then ${\mathbf v}{\mathbf v}^* \otimes X \in C_n$, so $X \in M_n^+$. It follows that $C_1 \subseteq M_n^{+}$, so $C_1 = M_n^{+}$, so the cones $\{C_m\}_{m=1}^{\infty}$ define an operator system on $M_n$.
		
	To prove uniqueness, assume that there exists another family of cones $\{D_m\}_{m=1}^{\infty}$ that define an operator system, such that $D_n = C_n$. It is clear that $C_m \subseteq D_m$ for all $m \in \bb{N}$, so we only need to prove the other inclusion. If $m \leq n$, let $X \in D_m$ and let $V : \bb{C}^m \rightarrow \bb{C}^n$ be an isometry (i.e., $V^*V = I$). Then $Y := ({\rm Ad}_{V^*} \otimes id_n)(X) \in D_n = C_n$, so $X = ({\rm Ad}_{V} \otimes id_n)(Y) \in C_m$. Thus $D_m \subseteq C_m$, so $D_m = C_m$ for $m \leq n$. If $m > n$ then we recall from \cite[Section 2.3]{Xthesis} the $k$-super minimal and $k$-super maximal operator system structures. In particular, it was shown that if two operator systems on $M_n$, defined by cones $\{C_m\}_{m=1}^{\infty}$ and $\{D_m\}_{m=1}^{\infty}$ respectively, are such that $C_m = D_m$ for $1 \leq m \leq n$, then $C_m = D_m$ for all $m \in \bb{N}$. See also \cite[Section 4]{JKPP11}.

	The ``only if'' direction of the final claim is trivial from the definition of super-homogeneity, and the ``if'' direction follows easily from the fact that $({\rm Ad}_A \otimes id_n)$ and $(id_m \otimes {\rm Ad}_B)$ commute. This completes the proof.
\end{proof}

We close this section with a result that shows that to determine complete positivity of a map from one operator system on $M_n$ to another, it is enough to look at the action of that map on the $n^{th}$ cone of the operator systems.
\begin{cor}\label{cor:cp_on_cn}
	Let $\Phi : M_n \rightarrow M_n$ and let $O_1(M_n)$ and $O_2(M_n)$ be operator systems defined by families of cones $\{C_m\}_{m=1}^{\infty}$ and $\{D_m\}_{m=1}^{\infty}$, respectively. Then $\Phi \in \cl{CP}(O_1(M_n),O_2(M_n))$ if and only if $(id_n \otimes \Phi)(C_n) \subseteq D_n$.
\end{cor}
\begin{proof}
	The ``only if'' implication follows trivially from the definition of $\cl{CP}(O_1(M_n),O_2(M_n))$. For the ``if'' implication, suppose $(id_n \otimes \Phi)(C_n) \subseteq D_n$. Fixing $m \in \bb{N}$ arbitrarily and applying $\sum_i {\rm Ad}_{A_i} \otimes id_n$ for $\{A_i\} \in M_{n,m}$ to both sides then gives
	\begin{align*}
		(id_m \otimes \Phi)(C_m) & = \bigcup_{\{A_i\} \in M_{n,m}} \left\{\sum_i({\rm Ad}_{A_i} \otimes \Phi)(C_n)\right\} \\
		& \subseteq \bigcup_{\{A_i\} \in M_{n,m}}\left\{\sum_i({\rm Ad}_{A_i} \otimes id_n)(D_n)\right\} \\
		& = D_m,
	\end{align*}
	where both of the above equalities follow from the form of the cones $\{C_m\}_{m=1}^{\infty}$ and $\{D_m\}_{m=1}^{\infty}$ guaranteed by Proposition~\ref{prop:cn_to_os}. It follows that $\Phi \in \cl{CP}(O_1(M_n),O_2(M_n))$, completing the proof.
\end{proof}

%%%%%%%%%%%%%%%%%%%%%%%%%%%%%%%%%%%%%%%%
\section{Right-CP-Invariant Cones as Operator Systems}\label{sec:opSysRight}
%%%%%%%%%%%%%%%%%%%%%%%%%%%%%%%%%%%%%%%%

In this section we establish a link between right-CP-invariant cones and operator systems. Our first result is in the same vein as some known results on mapping cones such as \cite[Theorem 1]{Sk10} and \cite[Theorem 1]{S09}. Here we prove an analogous statement for cones that are just right-CP-invariant.
\begin{prop}\label{prop:compose_right_cp}
	Let $\cl{C} \subseteq \cl{L}(M_n)$ be a right-CP-invariant cone. Then $\Psi^{\dagger} \circ \Phi \in \cl{CP}(M_n)$ for all $\Phi \in \cl{C}$ if and only if $\Psi \in \cl{C}^{\circ}$.
\end{prop}
\begin{proof}
	To prove the ``only if'' implication, suppose $\Psi \in \cl{L}(M_n)$ and $\Psi^{\dagger} \circ \Phi \in \cl{CP}(M_n)$ for all $\Phi \in \cl{C}$. Then $C_{\Psi^{\dagger} \circ \Phi} \in (M_n \otimes M_n)^+$ so
	\begin{align*}
		0 \leq \Tr(E C_{\Psi^{\dagger} \circ \Phi}) & = \Tr(E (id_n \otimes (\Psi^{\dagger} \circ \Phi))(E)) \\
		& = \Tr((id_n \otimes \Psi)(E) (id_n \otimes \Phi)(E)) = \Tr(C_{\Psi}C_{\Phi}) \quad \forall \, \Phi \in \cl{C},
	\end{align*}
	where we recall that $E := \sum_{i,j=1}^n {\mathbf e_i}{\mathbf e_j}^* \otimes {\mathbf e_i}{\mathbf e_j}^*$. It follows that $\Psi \in \cl{C}^{\circ}$. It is perhaps worth noting that the proof of this implication did not make use of right-CP-invariance of $\cl{C}$.
	
	To see why the ``if'' implication holds, assume $\Psi \in \cl{C}^{\circ}$. Then, because $\cl{C}$ is right-CP-invariant, it follows that for any $\Phi \in \cl{C}$ and $\Omega \in \cl{CP}(M_n)$ we have $\Phi \circ \Omega \in \cl{C}$ so
	\begin{align*}
		0 \leq \Tr(C_{\Psi}C_{\Phi \circ \Omega}) & = \Tr((id_n \otimes \Psi)(E) (id_n \otimes (\Phi \circ \Omega))(E)) \\
		& = \Tr((id_n \otimes (\Phi^\dagger \circ \Psi))(E) (id_n \otimes \Omega)(E)) = \Tr(C_{\Phi^\dagger \circ \Psi}C_{\Omega}).
	\end{align*}
	It follows via the Choi-Jamio\l kowski isomorphism that $C_{\Phi^\dagger \circ \Psi} \in (M_n \otimes M_n)^+$, so $\Phi^\dagger \circ \Psi \in \cl{CP}(M_n)$. Then $(\Phi^\dagger \circ \Psi)^{\dagger} = \Psi^\dagger \circ \Phi \in \cl{CP}(M_n)$, completing the proof.
\end{proof}

It is not difficult to verify that if $O(M_n)$ is any operator system, then $CP(M_n,O(M_n))$ is a right-CP-invariant cone. Similarly, $CP(O(M_n),M_n)$ is easily seen to be a closed left-CP-invariant cone. The main result of this section shows that these properties actually characterize the possible cones of completely positive maps to and from $M_n$, and furthermore that these cones uniquely determine $O(M_n)$.

Recall that $\cl{P}(M_n)$ denotes the cone of positive maps on $M_n$, $\cl{S}(M_n)$ denotes the cone of superpositive maps on $M_n$, and $C_{\cl{C}}$ denotes the cone of Choi matrices of maps from the cone $\cl{C}$.
\begin{thm}\label{thm:right_cp_invariant}
	Let $\cl{C} \subseteq \cl{L}(M_n)$ be a convex cone. The following are equivalent:
	\begin{enumerate}
		\item $\cl{C}$ is right-CP-invariant with $\cl{S}(M_n) \subseteq \cl{C} \subseteq \cl{P}(M_n)$.
		\item There exists an operator system $O_1(M_n)$, defined by cones $\{C_m\}_{m=1}^{\infty}$, such that $C_{\cl{C}} = C_n$.
		\item There exists an operator system $O_2(M_n)$ such that $\cl{C} = \cl{CP}(M_n,O_2(M_n))$.
		\item There exists an operator system $O_3(M_n)$ such that $(\cl{C}^{\circ})^{\dagger} = \cl{CP}(O_3(M_n),M_n)$.
	\end{enumerate}
	Furthermore, $O_1(M_n) = O_2(M_n)$ and is uniquely determined by $\cl{C}$, and $O_3(M_n)$ is uniquely determined by $\overline{\cl{C}}$ and can be chosen so that $O_3(M_n) = O_1(M_n)$.
\end{thm}
\begin{proof}
	We prove the result by showing that $(1) \Leftrightarrow (2)$, $(2) \Leftrightarrow (3)$, and $(2) \Leftrightarrow (4)$.
	
	To see that $(1) \Rightarrow (2)$, define $C_n := C_{\cl{C}}$. If $A \in M_n$ and $\Phi \in \cl{C}$ then
	\begin{align}\label{eq:cn_to_invariant}
		({\rm Ad}_A \otimes id_n)(C_{\Phi}) & = ({\rm Ad}_A \otimes \Phi)(E) = (id_n \otimes (\Phi \circ {\rm Ad}_{A^T}))(E) \in C_n,
	\end{align}
	where the second equality comes from Lemma~\ref{lem:left_right_choi} and the inclusion comes from the fact that $\cl{C}$ is right-CP-invariant, so $\Phi \circ {\rm Ad}_{A^T} \in \cl{C}$. The implication $(1) \Rightarrow (2)$ and uniqueness of $O_1$ then follow from Proposition~\ref{prop:cn_to_os}. The reverse implication $(2) \Rightarrow (1)$ also follows from the string of equalities~\eqref{eq:cn_to_invariant}, but this time we use the fact that $C_n$ is a cone defining an operator system to get the inclusion. The fact that $\cl{S}(M_n) \subseteq \cl{C} \subseteq \cl{P}(M_n)$ follows from the fact that for the minimal operator system on $M_n$, $C_n$ is the cone of block-positive operators and for the maximal operator system on $M_n$, $C_n$ is the cone of separable operators \cite[Theorem 5]{JKPP11}.
	
	To see that $(2) \Rightarrow (3)$, let $O_2(M_n) = O_1(M_n)$. We then have to show that if $C_{\cl{C}} = C_n$, then $\cl{C} = \cl{CP}(M_n,O_1(M_n))$. We already showed that $(2) \Rightarrow (1)$, so we know that $\cl{C}$ is right-CP-invariant. So if $\Phi \in \cl{C}$ then for any $X \in (M_n \otimes M_n)^{+}$ there exists $\Psi \in \cl{CP}(M_n)$ such that
	\begin{align*}
		(id_n \otimes \Phi)(X) = (id_n \otimes (\Phi \circ \Psi))(E) \in C_n,
	\end{align*}
	where the inclusion comes from $\cl{C}$ being right CP-invariant. It follows via Corollary~\ref{cor:cp_on_cn} that $\Phi \in \cl{CP}(M_n,O_1(M_n))$, so $\cl{C} \subseteq \cl{CP}(M_n,O_1(M_n))$. To see the opposite inclusion, simply note that if $\Phi \in \cl{CP}(M_n,O_1(M_n))$ then, because $E \in (M_n \otimes M_n)^{+}$, we have $C_{\Phi} = (id_n \otimes \Phi)(E) \in C_n = C_{\cl{C}}$, so $\Phi \in \cl{C}$. It follows that $\cl{C} = \cl{CP}(M_n,O_1(M_n))$.
	
	To establish uniqueness of $O_2$ (and simultaneously prove $(3) \Rightarrow (2)$), suppose that the cones $\{D_m\}_{m=1}^{\infty}$ define an operator system $O_2(M_n)$ such that $\cl{C} = \cl{CP}(M_n,O_2(M_n))$. Because $E \in (M_n \otimes M_n)^{+}$, we again have that $(id_n \otimes \Phi)(E) \in D_n$ for any $\Phi \in \cl{C}$, so $C_{\cl{C}} \subseteq D_n$. On the other hand by the equivalence of $(1)$ and $(2)$, $D_n = C_{\cl{C}^{\prime}}$ for some right-CP-invariant cone $\cl{C}^{\prime}$. If $\Phi \in \cl{C}^{\prime}$ then for any $X \in (M_n \otimes M_n)^{+}$ there exists $\Psi \in \cl{CP}(M_n)$ such that
	\begin{align*}
		(id_n \otimes \Phi)(X) = (id_n \otimes (\Phi \circ \Psi))(E) \in D_n,
	\end{align*}
	where the inclusion comes from $\cl{C}^{\prime}$ being right CP-invariant. It follows via Corollary~\ref{cor:cp_on_cn} that $\cl{C}^{\prime} \subseteq \cl{CP}(M_n,O_2(M_n))$. Then $\cl{C} \subseteq \cl{C}^{\prime} \subseteq \cl{CP}(M_n,O_2(M_n)) = \cl{C}$, so $\cl{C} = \cl{C}^{\prime}$ and hence $C_n = D_n$. Uniqueness now follows from Proposition~\ref{prop:cn_to_os}.
	
	The proof that $(2) \Leftrightarrow (4)$ mimics the proof that $(2) \Leftrightarrow (3)$ and makes use of the fact that $\Psi^{\dagger} \circ \Phi \in \cl{CP}(M_n)$ for all $\Phi \in \cl{C}$ if and only if $\Psi \in \cl{C}^{\circ}$ (Proposition~\ref{prop:compose_right_cp}). To see that $(2) \Rightarrow (4)$, let $O_3(M_n) = O_1(M_n)$. Then for any $\Psi \in \cl{C}^{\circ}$ and $\Phi \in \cl{C}$ we have $\Psi^{\dagger} \circ \Phi \in \cl{CP}(M_n)$, so $C_{\Psi^{\dagger} \circ \Phi} \in (M_n \otimes M_n)^{+}$. It follows that $(id_n \otimes \Psi^{\dagger})(C_{n}) \subseteq (M_n \otimes M_n)^{+}$. Corollary~\ref{cor:cp_on_cn} implies that $\Psi^{\dagger} \in \cl{CP}(O_3(M_n),M_n)$, so $(\cl{C}^{\circ})^{\dagger} \subseteq \cl{CP}(O_3(M_n),M_n)$. The opposite inclusion follows by simply reversing this argument.

	Uniqueness of $O_3$ (up to closure) and the implication $(4) \Rightarrow (2)$ follow similarly by the fact that $\Psi^{\dagger} \in \cl{CP}(O_{\cl{C}}(M_n),M_n)$ if and only if $\Psi^{\dagger} \circ \Phi \in \cl{CP}(M_n)$ for all $\Phi \in \cl{C}$ if and only if $\Psi \in \cl{C}^{\circ}$, where $O_{\cl{C}}(M_n)$ is an operator system with its $n^{th}$ cone $C_n := C_{\cl{C}}$.
\end{proof}

The equivalence of conditions (1) and (2) in Theorem~\ref{thm:right_cp_invariant} can be seen as providing a bijection between right-CP-invariant cones and operator systems on $M_n$. Given an operator system $O(M_n)$ defined by cones $\{C_m\}_{m=1}^{\infty}$, the associated right-CP-invariant cone is given via the maps associated to $C_n$ via the Choi-Jamio\l kowski isomorphism. In the other direction, given a right-CP-invariant cone, the associated operator system gets its $n^{th}$ cone from the Choi-Jamio\l kowski isomorphism and then gets its remaining cones via the construction given in Proposition~\ref{prop:cn_to_os}.

%%%%%%%%%%%%%%%%%%%%%%%%%%%%%%%%%%%%%%%%
\section{Mapping Cones as Operator Systems}\label{sec:opSys}
%%%%%%%%%%%%%%%%%%%%%%%%%%%%%%%%%%%%%%%%

Before introducing the main results of this section, we present a lemma that shows that the largest cone of completely positive maps between any two operator systems on $M_n$ is the cone of positive maps -- a result that follows from recent work on minimal and maximal operator systems \cite{PTT10,JKPP11}.
\begin{lemma}\label{lem:cp_sub_p}
	Let $O_1(M_n)$ and $O_2(M_n)$ be operator systems. Then $\cl{CP}(O_1(M_n),O_2(M_n)) \subseteq \cl{P}(M_n)$.
\end{lemma}
\begin{proof}
	Let $O_1(M_n)$ and $O_2(M_n)$ be defined by the families of cones $\{C_m\}_{m=1}^{\infty}$ and $\{D_m\}_{m=1}^{\infty}$, respectively. Let $\Phi \in \cl{L}(M_n)$ be such that $\Phi \notin \cl{P}(M_n)$. Because the smallest family of cones defining an operator system on $M_n$ are the separable operators and the largest such family of cones are the block-positive operators \cite[Theorem 5]{JKPP11}, we know that $I \otimes X \in C_n$ for all $X \in M_n^+$ and $D_n \subseteq P_n$. Because $\Phi \notin \cl{P}(M_n)$, there exists a particular $X \in M_n^{+}$ such that $\Phi(X) \notin M_n^+$. It is then easily verified that $I \otimes \Phi(X) \notin P_n$, so $(id_n \otimes \Phi)(I \otimes X) \notin D_n$. It follows that $\Phi \notin \cl{CP}(O_1(M_n),O_2(M_n))$, so $\cl{CP}(O_1(M_n),O_2(M_n)) \subseteq \cl{P}(M_n)$.
\end{proof}

The following result shows how the bijection inroduced by Theorem~\ref{thm:right_cp_invariant} works when the right-CP-invariant cone is in fact a mapping cone -- in this situation the associated operator system is super-homogeneous.

\begin{cor}\label{cor:mapping_cone_os}
	Let $\cl{C} \subseteq \cl{L}(M_n)$ be a closed, convex cone. The following are equivalent:
	\begin{enumerate}
		\item $\cl{C}$ is a mapping cone.
		\item There exists a super-homogeneous operator system $O_1(M_n)$, defined by cones\\
		$\{C_m\}_{m=1}^{\infty}$, such that $C_{\cl{C}} = C_n$.
		\item There exists a super-homogeneous operator system $O_2(M_n)$ such that\\
		$\cl{C} = \cl{CP}(M_n,O_2(M_n))$.
		\item There exists a super-homogeneous operator system $O_3(M_n)$ such that\\
		$(\cl{C}^{\circ})^{\dagger} = \cl{CP}(O_3(M_n),M_n)$.
		\item There exist super-homogeneous operator systems $O_4(M_n)$ and $O_5(M_n)$ such that \\
		$\cl{C} = \cl{CP}(O_4(M_n),O_5(M_n))$.
	\end{enumerate}
	Furthermore, $O_1(M_n) = O_2(M_n) = O_3(M_n)$ and is uniquely determined by $\cl{C}$.
\end{cor}
\begin{proof}
	The equivalence of $(1)$, $(2)$, $(3)$, and $(4)$ (and uniqueness of the corresponding operator systems) follows immediately from the corresponding statements of Theorem~\ref{thm:right_cp_invariant} and the fact that $\cl{C}$ is left-CP-invariant if and only if $(id_n \otimes {\rm Ad}_B)(C_{\cl{C}}) \subseteq C_{\cl{C}}$, which then gives super-homogeneity of the corresponding operator system via Proposition~\ref{prop:cn_to_os}.
	
	Because $M_n$ is a super-homogeneous operator system, it is clear that $(3) \Rightarrow (5)$. All that remains to do is prove that $(5) \Rightarrow (1)$. To this end, let $O_4(M_n)$ and $O_5(M_n)$ be super-homogeneous operator systems defined by families of cones $\{C_m\}_{m=1}^{\infty}$ and $\{D_m\}_{m=1}^{\infty}$, respectively. By the equivalence of conditions $(1)$ and $(2)$, we know that there exist mapping cones $\cl{C}^{\prime}$ and $\cl{C}^{\prime\prime}$ such that $C_{\cl{C}^{\prime}} = C_n$ and $C_{\cl{C}^{\prime\prime}} = D_n$. By Corollary~\ref{cor:cp_on_cn}, $(id_n \otimes \Phi)(C_n) \subseteq D_n$ if and only if $\Phi \in \cl{CP}(O_4(M_n),O_5(M_n))$, so it follows that $\Phi \circ \Psi \in \cl{C}^{\prime\prime}$ for all $\Psi \in \cl{C}^{\prime}$ if and only if $\Phi \in \cl{CP}(O_4(M_n),O_5(M_n))$. Right-CP-invariance of $\cl{CP}(O_4(M_n),O_5(M_n))$ now follows from left-CP-invariance of $\cl{C}^{\prime}$ and left-CP-invariance of $\cl{CP}(O_4(M_n),O_5(M_n))$ follows from left-CP-invariance of $\cl{C}^{\prime\prime}$. The fact that $\cl{CP}(O_4(M_n),O_5(M_n)) \subseteq \cl{P}(M_n)$ follows from Lemma~\ref{lem:cp_sub_p}.
\end{proof}

It is natural at this point to consider well-known mapping cones and ask what are the corresponding operator systems via the bijection of Corollary~\ref{cor:mapping_cone_os}. The mapping cone of standard completely positive maps $\cl{CP}(M_n)$ of course corresponds to the ``naive'' operator system with positive cones equal to the cones of positive semidefinite operators. It was shown in \cite{PTT10} that $\cl{S}(M_n) = \cl{CP}(M_n,OMAX(M_n))$, where $OMAX(M_n)$ is the maximal operator system structure on $M_n$. It follows that the operator system associated with the mapping cone $\cl{S}(M_n)$ is $OMAX(M_n)$, and the cones that define $OMAX(M_n)$ are exactly the cones of separable operators. It was similarly shown that $\cl{S}(M_n) = \cl{CP}(OMIN(M_n),M_n)$, where $OMIN(M_n)$ is the minimal operator system structure on $M_n$. It follows from condition~(4) of Corollary~\ref{cor:mapping_cone_os} (and the fact that $\cl{S}(M_n) = (\cl{P}(M_n)^\circ)^\dagger$) that the operator system associated with the mapping cone $\cl{P}(M_n)$ is $OMIN(M_n)$, and the cones that define $OMIN(M_n)$ are the cones of block-positive operators.

It was shown in \cite{JKPP11} that if $OMIN_k(M_n)$ and $OMAX_k(M_n)$ denote the super $k$-minimal and super $k$-maximal operator systems on $M_n$ \cite{Xthesis}, respectively, then we have that $\cl{P}_k(M_n) = \cl{CP}(M_n,OMIN_k(M_n))$ and $\cl{S}_k(M_n) = \cl{CP}(M_n,OMAX_k(M_n))$. Thus the operator systems associated with the mapping cones $\cl{P}_k(M_n)$ and $\cl{S}_k(M_n)$ are $OMIN_k(M_n)$ and $OMAX_k(M_n)$, respectively. Finally, consider the mapping cone of completely co-positive maps $\{\Phi \circ T : \Phi \in \cl{CP}(M_n)\}$. It is not difficult to see that the associated operator system is the one defined by the cones of operators with positive partial transpose -- i.e., the operators $X \in M_m \otimes M_n$ such that $(id_m \otimes T)(X) \geq 0$.

We close this section by considering what Corollary~\ref{cor:mapping_cone_os} says in the case when the mapping cone $\cl{C}$ is \emph{symmetric} -- that is, when $T \circ \Phi \circ T \in \cl{C}$ and $\Phi^{\dagger} \in \cl{C}$ whenever $\Phi \in \cl{C}$. The concept of symmetric mapping cones was seen to be important in \cite{S11}, and it is worth noting that all of the specific mapping cones considered so far, such as the cones of $k$-positive and completely co-positive maps, are in fact symmetric. It will be useful for us to define a linear operator $F \in M_n \otimes M_n$ by $F({\mathbf v} \otimes {\mathbf w}) = {\mathbf w} \otimes {\mathbf v}$ and extending linearly. The operator $F$ is sometimes called the \emph{swap} or \emph{flip} operator, and we observe that $F = F^T$.

\begin{thm}\label{thm:symmetric_cone_os}
	Let $\cl{C} \subseteq \cl{L}(M_n)$ be a convex mapping cone and let $O(M_n)$ be the operator system, defined by cones $\{C_m\}_{m=1}^{\infty}$, associated to $\cl{C}$ via the bijection of Corollary~\ref{cor:mapping_cone_os}. Then $\cl{C}$ is symmetric if and only if $C_n$ is closed under the transpose map and the map $X \mapsto FXF$.
\end{thm}
\begin{proof}
	The proof relies on Lemma~\ref{lem:left_right_choi} which tells us that $C_{T \circ \Phi^{\dagger} \circ T} = FC_{\Phi}F$, and \cite[Lemma 4]{S09} which tells us that $C_{T \circ \Phi \circ T} = C_{\Phi}^T$. Combining these two results shows that $C_{\Phi^{\dagger}} = FC_{\Phi}^{T}F$. It then follows immediately that $T \circ \Phi \circ T \in \cl{C}$ whenever $\Phi \in \cl{C}$ if and only if $C_n$ (which equals $C_{\cl{C}}$) is closed under the transpose map $T$. Similarly, $\Phi^{\dagger} \in \cl{C}$ whenever $\Phi \in \cl{C}$ if and only if $C_n$ is closed under the map $X \mapsto FX^{T}F$. The result follows.
\end{proof}

%%%%%%%%%%%%%%%%%%%%%%%%%%%%%%%%%%%%%%%%
\section{Semigroup Cones as Operator Systems}\label{sec:semigroups}
%%%%%%%%%%%%%%%%%%%%%%%%%%%%%%%%%%%%%%%%

Theorem~\ref{thm:right_cp_invariant} and Corollary~\ref{cor:mapping_cone_os} provide characterizations of completely positive maps to and from $M_n$, and completely positive maps between two different super-homogeneous operator systems on $M_n$. However, they say nothing about completely positive maps from a super-homogeneous operator system back into itself. Toward deriving a characterization for this situation, we will say that a cone $\cl{C} \subseteq \cl{L}(M_n)$ is a \emph{semigroup} if it is closed under composition -- i.e., if $\Phi \circ \Psi \in \cl{C}$ for all $\Phi,\Psi \in \cl{C}$. Notice that many of the standard examples of mapping cones, such as the $k$-positive maps and the $k$-superpositive maps, are semigroups (however, the cone of completely co-positive maps is not).

The following proposition is a generalization of the fact that $\Phi$ is $k$-positive if and only if $\Phi \circ \Psi$ is $k$-superpositive for all $k$-superpositive $\Psi$ \cite[Theorem 3.8]{SSZ09}. Note that it is similar to Proposition~\ref{prop:compose_right_cp}, but by using the fact that $\cl{C}$ is a semigroup instead of just right-CP-invariant or a mapping cone we are able to show that $\Phi^{\dagger} \circ \Psi \in \cl{C}^{\circ}$, not just that $\Phi^{\dagger} \circ \Psi \in \cl{CP}(M_n)$.

% NOTE: By reversing the role of C and C^o it seems maybe possible to remove the hypotheses of convexity and closedness of C. Not necessary for us, but maybe good to know.

\begin{prop}\label{prop:semigroup}
	Let $\cl{C} \supseteq \cl{CP}(M_n)$ be a closed convex cone semigroup. Then $\Phi \in \cl{C}$ if and only if $\Phi^{\dagger} \circ \Psi \in \cl{C}^{\circ}$ for all $\Psi \in \cl{C}^{\circ}$.
\end{prop}
\begin{proof}
	To show the ``only if'' direction, it is enough to show that $\Tr(C_{\Phi^{\dagger} \circ \Psi}C_{\Omega}) \geq 0$ for all $\Omega \in \cl{C}$. To this end, simply note that
	\begin{align*}
		\Tr(C_{\Phi^{\dagger} \circ \Psi}C_{\Omega}) = \Tr(C_{\Psi}C_{\Phi \circ \Omega}) \geq 0,
	\end{align*}
	where the final inequality follows from the fact that $\Phi,\Omega \in \cl{C}$ so $\Phi \circ \Omega \in \cl{C}$.
	
	To see the ``if'' direction, suppose $\Phi^{\dagger} \circ \Psi \in \cl{C}^{\circ}$ for all $\Psi \in \cl{C}^{\circ}$. Then, because $id_n \in \cl{CP}(M_n) \subseteq \cl{C}$, we have
	\begin{align*}
		0 \leq \Tr(C_{\Phi^{\dagger} \circ \Psi}E) & = \Tr(C_{\Psi}C_{\Phi}) \quad \forall \, \Psi \in \cl{C}^{\circ}.	
	\end{align*}
	It follows that $\Phi \in \cl{C}^{\circ \circ} = \cl{C}$.
\end{proof}

If $O(M_n)$ is an operator system defined by cones $\{C_m\}_{m=1}^{\infty}$, then the dual cones $\{C_m^{\circ}\}_{m=1}^{\infty}$ define an operator system as well, which we will denote $O^{\circ}(M_n)$. For simplicity, we will only consider this operator system as a family of dual cones, in keeping with our focus throughout the preceding portion of the paper, and not the associated dual operator space structure. The interested reader is directed to \cite{BM10} for a more thorough treatment of dual operator systems. It is easily verified that $O(M_n)$ is super-homogeneous if and only if $O^{\circ}(M_n)$ is super-homogeneous, and the ``naive'' operator system on $M_n$ is easily seen to be self-dual: $M_n^{\circ} = M_n$. By the duality of the cones of $k$-positive maps and $k$-superpositive maps we know that $OMIN_{k}^{\circ}(M_n) = OMAX_k(M_n)$ and $OMAX_{k}^{\circ}(M_n) = OMIN_k(M_n)$.

We now consider what types of cones can be completely positive from a super-homogeneous operator system back into itself. By using \cite[Theorem 5]{JKPP11} and the fact that $\cl{P}_k(M_n)$ is a semigroup, it is not difficult to see that $\cl{CP}(OMIN_k(M_n)) = \cl{P}_k(M_n)$. By using \cite[Theorem 3.8]{SSZ09} we can similarly see that $\cl{CP}(OMAX_k(M_n)) = \cl{P}_k(M_n)$, so we can't possibly hope for a uniqueness result as strong as that of Theorem~\ref{thm:right_cp_invariant} or Corollary~\ref{cor:mapping_cone_os} in this setting. Nonetheless, we have the following result, which shows that duality plays a strong role here and the fact that $\cl{CP}(OMIN_k(M_n)) = \cl{CP}(OMAX_k(M_n))$ follows from the duality of $OMIN_k(M_n)$ and $OMAX_k(M_n)$. Furthermore, there is a unique operator system that gives the cone $\cl{CP}(O(M_n))$ that is ``large enough'' to contain $(M_n \otimes M_n)^{+}$ as a subset of its $n^{th}$ cone -- in this case it is $OMIN_k(M_n)$.

\begin{thm}\label{thm:cp_semigroup}
	Let $\cl{C} \subseteq \cl{L}(M_n)$ be a convex cone. The following are equivalent:
	\begin{enumerate}
		\item $\cl{C}$ is a semigroup cone with $\cl{CP}(M_n) \subseteq \cl{C} \subseteq \cl{P}(M_n)$.
		\item There exists a super-homogeneous operator system $O(M_n)$ such that $\cl{C} = \cl{CP}(O(M_n))$.
	\end{enumerate}
	Additionally, $\cl{CP}(O^{\circ}(M_n)) = \overline{\cl{CP}(O(M_n))}^{\dagger}$ and $O(M_n)$ can be chosen so that its $n^{th}$ cone $C_n = C_{\cl{C}}$. Furthermore, $O(M_n)$ is unique up to the condition $(M_n \otimes M_n)^{+} \subseteq C_n$.
\end{thm}

\begin{proof}
	We first prove that $(2) \Rightarrow (1)$. Let $\{C_m\}_{m=1}^{\infty}$ be the cones associated with the operator system $O(M_n)$. If $X \in C_m$ and $\Phi,\Psi \in \cl{CP}(O(M_n))$ then $(id_m \otimes \Phi)(X) \in C_m$. But then applying $id_m \otimes \Psi$ shows $(id_m \otimes (\Psi \circ \Phi))(X) \in C_m$ as well, so it follows that $\Psi \circ \Phi \in \cl{CP}(O(M_n))$ and thus $\cl{CP}(O(M_n))$ is a semigroup. Because $O(M_n)$ is super-homogeneous, we know that ${\rm Ad}_B \in \cl{CP}(O(M_n))$ for all $B \in M_n$, and so $\cl{CP}(M_n) \subseteq \cl{CP}(O(M_n))$. To see that $\cl{CP}(O(M_n)) \subseteq \cl{P}(M_n)$, simply use Lemma~\ref{lem:cp_sub_p}.

	To see that $(1) \Rightarrow (2)$, we argue much as we did in Theorem~\ref{thm:right_cp_invariant}. It is clear, via the Choi-Jamio\l kowski isomorphism, that $S_n \subseteq C_{\cl{C}} \subseteq P_n$. Now note that $\cl{C}$ is left- and right-CP-invariant (but perhaps not a mapping cone because it may not be closed) because $\Phi \circ \Psi \in \cl{C}$ for any $\Phi \in \cl{C}$ and $\Psi \in \cl{CP}(M_n) \subseteq \cl{C}$ (and similarly for composition on the left by $\Psi \in \cl{CP}(M_n)$). Thus, if $A \in M_n$ and $\Phi \in \cl{C}$ then
	\begin{align*}
		({\rm Ad_{A} \otimes {\rm Ad}_{B}})(C_{\Phi}) & = ({\rm Ad}_A \otimes ({\rm Ad}_{B} \circ \Phi))(E) = (id_n \otimes ({\rm Ad}_{B} \circ \Phi \circ {\rm Ad}_{A^T}))(E) \in C_{\cl{C}},
	\end{align*}
	where the second equality comes from Lemma~\ref{lem:left_right_choi}. It follows from Proposition~\ref{prop:cn_to_os} that there exists a super-homogeneous operator system $O(M_n)$, defined by cones $\{C_m\}_{m=1}^{\infty}$, such that $C_n = C_{\cl{C}}$. Because $\cl{C}$ is a semigroup, it follows that $(id_n \otimes (\Phi \circ \Psi))(E) \in C_n$ for any $\Phi,\Psi \in \cl{C}$. Then $(id_n \otimes \Phi)(C_\Psi) \in C_n$, so $(id_n \otimes \Phi)(C_n) \subseteq C_n$, which implies $\cl{C} \subseteq \cl{CP}(O(M_n))$ by Corollary~\ref{cor:cp_on_cn}. To see the other inclusion, note that $id_n \in \cl{CP}(M_n)$, so $id_n \in \cl{C}$. It follows that $(id_n \otimes id_n)(E) = E \in C_n$. Thus, if $\Phi \in \cl{CP}(O(M_n))$ then $(id_n \otimes \Phi)(E) \in C_n = C_{\cl{C}}$, so $\Phi \in \cl{C}$, which implies that $\cl{C} = \cl{CP}(O(M_n))$.
	
	To see the claim about $\cl{CP}(O^{\circ}(M_n))$, suppose that $\cl{CP}(M_n) \subseteq \cl{C} \subseteq \cl{P}(M_n)$ is a closed convex cone semigroup. Then for any $\Phi,\Psi \in \cl{C}^{\circ} \subseteq \cl{CP}(M_n)$ and $\Omega \in \cl{C}$ we have $\Tr(C_{\Phi \circ \Psi}C_{\Omega}) = \Tr(C_{\Psi}C_{\Phi^{\dagger} \circ \Omega})$. We know from Proposition~\ref{prop:semigroup} that $\Phi^{\dagger} \circ \Omega = (\Omega^\dagger \circ \Phi)^{\dagger} \in (\cl{C}^{\circ})^{\dagger} \subseteq \cl{CP}(M_n)^{\dagger} = \cl{CP}(M_n) \subseteq \cl{C}$. It follows that $\Tr(C_{\Psi}C_{\Phi^{\dagger} \circ \Omega}) \geq 0$, so $\Phi \circ \Psi \in \cl{C}^{\circ}$, which implies that $\cl{C}^{\circ}$ is also a semigroup.
	
	Now by repeating our argument from earlier, we see from Proposition~\ref{prop:cn_to_os} that there is an operator system on $M_n$ defined by the cone $C_n := C_{\cl{C}^{\circ}} = C_{\cl{C}}^{\circ}$, and this is the dual operator system $O^{\circ}(M_n)$ of the operator system defined by $C_{\cl{C}}$. For any $\Phi \in \cl{C}^{\circ \circ}, \Psi \in \cl{C}^{\circ}$, we have $(id_n \otimes \Phi^{\dagger})(C_{\Psi}) = C_{\Phi^{\dagger} \circ \Psi} \in C_{\cl{C}^{\circ}}$ by Proposition~\ref{prop:semigroup}. It follows via Corollary~\ref{cor:cp_on_cn} that $(\cl{C}^{\circ \circ})^{\dagger} = \overline{\cl{C}}^{\dagger} \subseteq \cl{CP}(O^{\circ}(M_n))$. To see the other inclusion, suppose $\Phi \in \cl{CP}(O^{\circ}(M_n))$. Then $\Phi \circ \Psi \in \cl{C}^{\circ}$ for all $\Psi \in \cl{C}^{\circ}$, so Proposition~\ref{prop:semigroup} tells us that $\Phi \in (\cl{C}^{\circ \circ})^{\dagger} = \overline{\cl{C}}^{\dagger}$. It follows that $\overline{\cl{C}}^{\dagger} = \cl{CP}(O^{\circ}(M_n))$.
	
	Finally, to see the uniqueness condition, suppose that the cones $\{D_m\}_{m=1}^{\infty}$ define an operator system $O_2(M_n)$ such that $\cl{C} = \cl{CP}(O(M_n)) = \cl{CP}(O_2(M_n))$, where $O(M_n)$ is the operator system with $n^{th}$ cone $C_n := C_{\cl{C}}$ already introduced. We furthermore require that $(M_n \otimes M_n)^{+} \subseteq D_n$, and in particular that $E \in D_n$. Then $(id_n \otimes \Phi)(E) \in D_n$ for any $\Phi \in \cl{C}$, so $C_{\cl{C}} \subseteq D_n$. On the other hand by the equivalence of $(1)$ and $(2)$, $D_n = C_{\cl{C}^{\prime}}$ for some semigroup cone $\cl{C}^{\prime}$. If $\Phi \in \cl{C}^{\prime}$ then for any $X \in D_n$ there exists $\Psi \in \cl{C}^{\prime}$ such that
	\begin{align*}
		(id_n \otimes \Phi)(X) = (id_n \otimes (\Phi \circ \Psi))(E) \in D_n,
	\end{align*}
	where the inclusion comes from $\cl{C}^{\prime}$ being a semigroup. It follows via Corollary~\ref{cor:cp_on_cn} that $\cl{C}^{\prime} \subseteq \cl{CP}(O_2(M_n))$. Then $\cl{C} \subseteq \cl{C}^{\prime} \subseteq \cl{CP}(O_2(M_n)) = \cl{C}$, so $\cl{C} = \cl{C}^{\prime}$ and hence $C_n = D_n$. Uniqueness now follows from Proposition~\ref{prop:cn_to_os}.
\end{proof}

It is worth noting that if $\cl{C}$ is closed and condition (1) of Theorem~\ref{thm:cp_semigroup} holds, then $\cl{C}$ is necessarily a mapping cone. It follows that if $O(M_n)$ is a super-homogeneous operator system defined by closed cones then $\cl{CP}(O(M_n))$ is always a mapping cone (which can also be seen from Corollary~\ref{cor:mapping_cone_os}), although the converse does not hold. That is, there exist mapping cones $\cl{C}$ such that there is no operator system $O(M_n)$ with $\cl{C} = \cl{CP}(O(M_n))$ -- the simplest example being the mapping cone of completely co-positive maps.

\vspace{0.1in}

\noindent{\bf Acknowledgements.} Thanks are extended to Vern Paulsen for valuable suggestions and comments on an early draft. N.J. was supported by an NSERC Canada Graduate Scholarship and the University of Guelph Brock Scholarship.

%%%%%%%%%%%%%%

\end{document}